\documentclass[11pt]{amsart}

\usepackage{AFBoix2015}
\usepackage{comment}

\begin{document}

\author[A.\,F.\,Boix]{Alberto F.\,Boix}

\address{Department of Mathematics, Ben-Gurion University of the Negev, P.O.B. 653 Beer-Sheva 8410501, Israel.}
\email{fernanal@post.bgu.ac.il}
\thanks{A.F.B. was supported by Israel Science Foundation (grant No. 844/14) and Spanish Ministerio de Econom\'ia y Competitividad MTM2016-7881-P}

\author[M.\,P.\,Noordman]{Marc Paul Noordman}
\address{Bernoulli Institute, University of Groningen, P.O. Box 407, 9700 AG Groningen, The Netherlands.}
\email{m.p.noordman@rug.nl}



\author[J.\,Top]{Jaap Top}

\address{Bernoulli Institute, University of Groningen, P.O. Box 407, 9700 AG Groningen, The Netherlands.}
\email{j.top@rug.nl}

\date{\today}

\title[The level of pairs of polynomials]{The level of pairs of polynomials}



\maketitle

\begin{abstract}
Given a polynomial $f$ with coefficients in a field of prime characteristic $p$, 
it is known that there exists a differential operator that raises $1/f$ to its $p$th power. We first discuss a relation
between the `level' of this differential operator and the
notion of `stratification' in
the case of hyperelliptic curves. 

Next we extend the notion of
level to that of a pair of
polynomials.
We prove some basic properties 
and we compute this level in certain special cases. 
In particular we present
examples of polynomials $g$ and $f$ such that there is no differential
operator raising $g/f$ to its $p$th power.
\end{abstract}

\section{Introduction}

Let $k$ be any perfect field and $R=k[x_1,...,x_d]$ its polynomial ring in $d$ variables. In this case it is known \cite[IV, Th\'eor\`eme 16.11.2]{EGAIV} that the ring $\DD$ of $k$--linear differential operators on $R$ is the $R$-algebra (which we take here as a definition)
\[
\DD:=R \left\langle D_{x_i,t} \mid i=1,\ldots,d \mbox{ and } t\geq 1 \right\rangle\subseteq\mathrm{End}_{k}(R), 
\]
generated by the operators $D_{x_i,t}$, defined as
\[
D_{x_i,t}(x_j^s)=\begin{cases}
\binom{s}{t}x_i^{s-t},\text{ if }i=j\text{ and }s\geq t,\\
0,\text{ otherwise }.\end{cases}
\]
For a non-zero $f\in R$, let $R_f$ be the localization of $R$ at $f$; the natural action of $\DD$ on $R$ extends to $R_f$ in such a way that $R_f=\DD\frac{1}{f^m}$, for some $m\geq 1$. Whilst there are examples of $m>1$ in characteristic $0$ (e.g. \cite[Example 23.13]{Twentyfourhours}), in positive characteristic one may always take $m=1$ (\cite[Theorem 3.7 and Corollary 3.8]{AlvarezBlickleLyubeznik2005}). This is shown by proving the existence of a differential operator $\delta\in\DD$ such that $\delta(1/f)=1/f^p$, i.e., $\delta$ acts as Frobenius on $1/f$. We want to mention here that the existence of this differential operator was used as key ingredient in \cite{BBLSZ14} to prove that local cohomology modules over smooth $\Z$--algebras have finitely many associated primes. On the other hand, the fact that $R_f$ is generated by $1/f$ as $\DD$--module remains valid for more
general classes of rings $R$: the interested reader may consult \cite[Theorems 4.1 and 5.1]{AlvarezBlickleLyubeznik2005}, \cite[Theorem 3.1]{Hsiao2012}, \cite[Corollary 2.10 and Remark 2.11]{TakTak08} and \cite[Theorem 4.4]{AlvarezMontanerHunekeBetancourt2017} for details. We will suppose that $k=\mathbb{F}_p$ and fix an algebraic closure $\overline{k}$ of $k$ from now on.

For an integer $e\geq 0$, let $R^{p^e}\subseteq R$ be the subring of all the $p^e$ powers of all the elements of $R$ and
set $\DD^{(e)}:=\End_{R^{p^e}} (R)$, the ring of $R^{p^e}$-linear ring-endomorphism of $R$. Since $R$ is a finitely generated $R^p$-module, by \cite[1.4.8 and 1.4.9]{Yekutiely1992}, it is 
\[
\DD=\bigcup_{e\geq 0}\DD^{(e)}.
\label{filter0}
\]
Therefore, for $\delta\in\DD$, there exists $e\geq 0$ such that $\delta\in \DD^{(e)}$ but $\delta\not\in \DD^{(e')}$ for any $e'<e$. This number $e$ is called the level of $\delta$. For a polynomial $f$, the level is defined
as the lowest level of an operator $\delta$ such that $\delta(1/f)=1/f^p$.

The level of a polynomial has been studied in \cite{AlvarezBlickleLyubeznik2005} and \cite{BoixDeStefaniVanzo2015}. In \cite{BoixDeStefaniVanzo2015}, an algorithm is given to compute the level and a number of examples are exhibited. Moreover, if $f$ is a cubic smooth homogeneous polynomial defining an elliptic curve $\mathcal{C}=V(f)=\{(x:y:z)\in \mathbb{P}^2_k: f(x,y,z)=0\},$ then the level of $f$ characterizes the supersingularity of $\mathcal{C}$ in the following way:

\begin{teo}(\cite[Theorem 1.1]{BoixDeStefaniVanzo2015})
Let $f\in R$ be a cubic homogeneous polynomial such that $\mathcal{C}=V(f)$ is an elliptic curve over $k$. Denote by $e$ the level of $f$. Then
\begin{enumerate}[(i)]
\item $\mathcal{C}$ is ordinary if and only if $e=1$.
\item $\mathcal{C}$ is supersingular if and only if $e=2$.
\end{enumerate}
\label{intro1}
\end{teo}

This result was generalized for hyperelliptic curves of arbitrary genus
$g\geq 2$; indeed, let $\mathcal{C}:=\{(x:y:z)\in \mathbb{P}^2_k: f(x,y,z)=0\}$, where $f$ is a homogeneous polynomial of degree $2g+1$ defined over $k$. If
$\mathrm{Jac}(\mathcal{C})$ denotes its Jacobian, then it is well known \cite[Proposition of page 60]{MumfordAbelianbook} that, for any integer $n>0$,
\[
\mathrm{Jac}(\mathcal C)[n](\overline k)=\begin{cases}
(\mathbb{Z}/n\mathbb{Z})^{2g} \text{  if } \mathrm{char}(k)\nmid n,\\
(\mathbb{Z}/p^m\mathbb{Z})^i \text{  if }n=p^m,\ p=\mathrm{char}(k)\text{  and }m>0,
\end{cases}
\]
where $i$ can take every value in the range $0\le i \le g$, and is called the
$p$-rank of $\mathcal C$. For the convenience of the reader, we recall here the following standard terminology:

\begin{df}The curve $\mathcal{C}$ is said to be ordinary if its
$p$-rank is maximal, i.e., equal to the genus of $\mathcal{C}$. The curve $\mathcal{C}$ is said to be supersingular (resp.\
superspecial) if $\mathrm{Jac}(\mathcal{C})$ is isogenous (resp.\ isomorphic)
over $\overline{k}$ to the product of $g$ supersingular elliptic curves.
If $\mathcal{C}$ is supersingular then the $p$-rank of $\mathcal{C}$ equals $0$,
however the converse of this statement does not hold.
\end{df}
The generalization of Theorem \ref{intro1} reads as follows
\cite[Theorems 1.3, 3.5 and 3.9]{BlancoBoixFordhamYilmaz2018}:

\begin{teo}\label{intro2}
Let $f\in R$ be a homogeneous polynomial in three variables and of degree $2g+1$, such that $\mathcal{C}\cong V(f)\subset \mathbb{P}^2$ defines a hyperelliptic curve over $\overline{k}$ of genus $g$. Denote by $e$ the level of $f$. Assume $p>2g^2-1$. Then
\begin{itemize}
        \item[(i)] $e=2$ if $\mathcal C$ is ordinary,
        \item[(ii)] $e>2$ if $\mathcal C$ is supersingular but not
          superspecial.
    \end{itemize}
\end{teo}
We also want to mention here that the level of a polynomial $f$ 
is closely related to the so--called Hartshorne--Speiser--Lyubeznik--Gabber number of the pair $(R,f)$, and that this number can be explicitly calculated using Macaulay2, see \cite[\S 4.4]{Boixet9authorsM2} for further information. On the other hand, one can also calculate the level of $f$ in terms of $F$--jumping numbers \cite[Proposition 6]{Fordham18}.

The goal of this paper is to introduce and study the level of a pair
of polynomials. Given $f,g$ polynomials
defined over $\mathbb{F}_p$, one may ask whether there is a differential operator
$\delta \in \mathcal D_R$ mapping $g/f$ to $(g/f)^p$. Such an operator exists when $g=1$ by \cite[Theorem 3.7 and Corollary 3.8]{AlvarezBlickleLyubeznik2005}, and
more generally, when $f$ itself has level one, as pointed out in
\cite{BoixDeStefaniVanzo2015}. Keeping in mind all of this, it seems natural to define the level
of $g$ and $f$ as
\[
\level (g,f):=\inf\{e\geq 0:\ \exists\delta\in
\mathcal D^{(e)}\text{ such that }\delta(g/f)=(g/f)^p\}.
\]
As we already mentioned, our goal in this paper is to study this notion, and to calculate it in several interesting examples. 

Part of our motivation for introducing it comes from \cite{Singh2017}, where the author gave a conceptual proof of a polynomial identity obtained in \cite[Lemma 3.1]{Singh2000} using hypergeometric series algorithms.
This polynomial identity, and the corresponding results obtained by Singh concerning associated primes of local cohomology modules \cite{Singh2000} were the basis of \cite{LyubeznikSinghWaltherdeterminantal}, where the authors proved, among other remarkable results, that local cohomology modules $H_{I_t (X)}^ k (\Z[X])$ are rational vector spaces for any $k>height (I_t(X))$, where $X$ is a matrix of indeterminates, and $I_t (X)$ is the ideal of size $t$ minors of this matrix \cite[Theorem 1.2]{LyubeznikSinghWaltherdeterminantal}. 
The proof presented in \cite{Singh2017} used as key ingredient certain differential operators defined over the integers that, modulo a prime $p$, act as the Frobenius endomorphism on quotients of polynomials \cite[page 244]{Singh2017}. 

Another motivation comes from \cite{BrennerNBJdifferential}, where the authors use higher order differential operators to measure various kind of singularities in all characteristics.
These higher order operators also play a key role in recent developments in the study of symbolic powers of ideals (see \cite{DeStefaniGrifoJeffries2018} and \cite[Section 10]{BrennerNBJdifferential} for details). We hope that the calculation of the level of a pair of polynomials might help in the understanding of these differential operators. 
The interplay between differential operators over the integers and their reduction modulo a prime $p$ (which is a delicate issue, see \cite[Section 6]{Jeffries2019} for details) was a key technical ingredient to prove in \cite[Theorem 3.1]{BBLSZ14} that local cohomology modules over $\Z$ can have $p$--torsion for at most finitely many primes $p.$

Now, we provide a more detailed overview of the contents of this
manuscript for the convenience of the reader; first of all, in
Section \ref{stratification and level section} we give some
connection between being stratified for a nonlinear differential
equation and the level of a polynomial in the case of
hyperelliptic curves. Second, in Section
\ref{the level of a pair: basic properties}, we formally define
the level of a pair of polynomials, listing some of the properties
it satisfies. In Section \ref{some concrete calculations
of the level of a pair}, we focus on specific calculations when
$f$ and $g$ are both homogeneous polynomials; in particular, we will
show, among other things, that $\level (g,f)$ is, in general, not
finite (see Proposition \ref{maybe of infinite level}). We end this paper by raising some open questions to stimulate
further research on this subject.
\section*{Acknowledgement}
This research started when the first named author visited the university of Groningen in the
Fall of 2017. We thank Marius van der Put for valuable discussions and for his interest in this work.

\section{Stratified differential equations and hyperelliptic curves}\label{stratification and level section}
The notion of stratification for nonlinear differential
equations was introduced in \cite{VanderPutTop2015}; we briefly recall
it here. Let
$C\supseteq\mathbb{F}_p$ be an algebraically closed field, let
$C(z)$ be the one variable differential field extension
of $C$ with derivation $\frac{d}{dz}$ and let $K$ be a finite separable extension of $C(z)$. Consider
the differential equation $f(y',y)=0$, where $f\in K[S,T]$ is an
absolutely irreducible polynomial such that the image $d$ of
$df/dS$ in $K[S,T]/(f)$ is nonzero; the differential algebra
$A:=K[y',y,1/d]$ is given by the derivation $D$ with $D(z)=1$ and
$D(y)=y'$. One says that $f(y',y)=0$ is \textbf{stratified} if and
only if $D^p=0$ \cite[Theorem 1.1]{VanderPutTop2015}; it was also
proved in \cite[Proposition 2.3]{VanderPutTop2015} that, if
$p\geq 3$ and $f$ is the defining equation of an elliptic curve $E$, then
$f(y',y)=0$ is stratified if and only if $E$ is supersingular, which is
equivalent to say, by Theorem \ref{intro1}, that the
homogeneous polynomial corresponding to $f$ has level
two. Keeping in mind these characterizations, one may ask what is the
connection between being stratified and the level of a polynomial. For this we will use the next technical result,
involving among other notions the $a$-number of (the Jacobian variety of) a curve $X$ of genus $g$.
This number equals the dimension of the kernel of the
Cartier-Manin matrix associated to $X$. Many properties
of it are discussed in the textbook \cite{LiOort1998};
the $p$-rank $f_X$ and the $a$-number $a_X$
satisfy $f_X+a_X\leq g$.
Here equality does not hold in general, but
$a_X=0\Leftrightarrow f_X=g\Leftrightarrow X\;\text{is ordinary}$, and $a_X=g\Leftrightarrow X\;\text{is superspecial}$ (see \cite[Theorem 2]{Oort1975} and \cite[Theorem 4.1]{Nygaard1981} for the latter).

\begin{prop}\label{the autonomous hyperelliptic equation}
Given an algebraically closed field $k$ of prime characteristic $p\geq 3,$ consider the hyperelliptic curve
$\mathcal{H}$ of genus $g\geq 1$ defined by the equation
$y^2=h(x),$ where $h(x)\in k[x]$ is squarefree
and has degree $2g+1.$ 
The following statements are equivalent.
\begin{enumerate}[(i)]

\item $\mathcal{H}$ is not ordinary.

\item There exist $a_0,a_1,\ldots, a_{g-1}\in k$ with $a_j\neq 0$ for at least one $j$,
such that 
the differential equation
\[
(x')^2=\frac{h(x)}{(a_{g-1}x^{g-1}+\ldots+a_1 x+a_0)^2}
\]
is stratified.

\item The $a$-number of the Jacobian of $\mathcal{H}$ 
is not zero.
\end{enumerate}
\end{prop}

\begin{proof}
Let $\mathcal{C}'$ be the modified Cartier operator defined
in \cite[Definition 2.1']{Yui1978}; by the argument pointed out
in \cite[page 312]{VanderPutTop2015}, our differential equation
is stratified if and only the differential
form $\omega:=((a_{g-1}x^{g-1}+\ldots+a_1 x+a_0)/y)dx$ is exact, which
is equivalent to the condition $\mathcal{C}'(\omega)=0$. Our goal now is
to write down this condition in terms of the basis of differentials
$\omega_i:=(x^{i-1}/y)dx$ ($1\leq i\leq g$); it is easy to see that $\mathcal{C}'(\omega)=0$
if and only if
\[
\sum_{i=1}^g a_{i-1}^{1/p}\mathcal{C}'(\omega_i)=0.
\]
Now, if one writes $h(x)^{(p-1)/2}=\sum_{j=0}^N
c_j x^j$, (where $N=((p-1)/2)(2g+1)$) then one has
\cite[page 381]{Yui1978} that
\[
\mathcal{C}'(\omega_i)=\sum_{j=1}^g c_{jp-i}\omega_j,
\]
and therefore one ends up with the following equality:
\[
\sum_{j=1}^g \left(\sum_{i=1}^g a_{i-1}^{1/p}c_{jp-i}\right)\omega_j=0.
\]
Equivalently, since the $\omega_j$'s are $k$--linearly
independent,  for any $1\leq j\leq g,$
\[
\sum_{i=1}^g a_{i-1}^{1/p}c_{jp-i}=0.
\]
Summing up, if one denotes by $v$ the column vector
$(a_0^{1/p},\ldots, a_{g-1}^{1/p})$ and by $C$ the Cartier--Manin
matrix of the hyperelliptic curve $y^2=h(x)$
\cite[Definition 2.2]{Yui1978}, one has that our differential
equation is stratified if and only if $C\cdot v=0$, which, by \cite[Theorem 3.1]{Yui1978}, 
is equivalent to the statement that
the hyperelliptic curve $y^2=h(x)$ is not ordinary. This proves the equivalence between (i) and (ii); finally, the equivalence between (i) and (iii) follows immediately from the fact that the $a$-number of $\mathrm{Jac}(\mathcal{H})$ 
equals the corank of the Cartier--Manin matrix of $\mathcal{H}$ \cite[5.2.8]{LiOort1998}.
\end{proof}

Combining Proposition \ref{the autonomous hyperelliptic equation}
with Theorem \ref{intro2}, we obtain the following result.

\begin{cor}\label{stratification and level for higher genus}
Preserving the assumptions and notations of Proposition \ref{the autonomous hyperelliptic equation}, let
$g\geq 2,$ $p>2g^2-1,$ and let $f=y^2z^{2g-1}-z^{2g+1}h(x/z)$. If $\level (f)\geq 3$, then there are $a_0,a_1,\ldots, a_{g-1}\in k$ with $a_j\neq 0$ for at least one $j$ such that the equation
\[
(x')^2=\frac{h(x)}{(a_{g-1}x^{g-1}+\ldots+a_1 x+a_0)^2}
\]
is stratified.
\end{cor}
The next examples  illustrate some of the results obtained above. 

\begin{ex}\label{level does not characterize stratified}
Given $0\neq b\in\mathbb{F}_p,$ and $p>7,$ consider the equation
\begin{equation}\label{stratified equation example}
(x')^2=\frac{x^5+b}{(a_1x+a_0)^2},
\end{equation}
and assume that $p\equiv 3\pmod 5$ (e.g. $p=13$). The hyperelliptic curve of genus two $\mathcal{H}$ defined
by $y^2=x^5+b$ has the following Cartier--Manin matrix:
\[
\begin{pmatrix}
0& 0\\ c& 0
\end{pmatrix},\text { where }c:=\binom{(p-1)/2}{(2p-1)/5}b^{(p-3)/10}.
\]
In particular, $\mathcal{H}$ is not ordinary. In this case, $\mathcal{H}$
is supersingular (but not superspecial) and therefore $\level (y^2z^3-x^5-bz^5)\geq 3$ by \cite[Corollary 3.10]{BlancoBoixFordhamYilmaz2018}. The equation \eqref{stratified equation example} is stratified, if and only if  $a_1=0$, as follows from the fact that the differential form $dx/y$ is in the kernel of the Cartier operator, whereas  for $a_1\neq 0$ the form $a_0dx/y+a_1xdx/y$ is not in the kernel.

Assume that $p\equiv 4\pmod 5$ (e.g. $p=19$). In this case, by either \cite[Theorem 2]{Valentini1995} or \cite[Corollary of page 12]{KodamaWashio1986}, $\mathcal{H}$ is superspecial and therefore 
\eqref{stratified equation example} is stratified for any value of $a_1,a_0.$ In this case, $\level (y^2z^3-x^5-bz^5)\geq 3$ by \cite[Example 4.4]{BlancoBoixFordhamYilmaz2018}. In contrast, where $p\equiv 1\pmod 5$ (e.g. $p=11$), one can easily check that $\mathcal{H}$ is ordinary (this also follows from \cite[Theorem 3]{KodamaWashio1986}) and therefore 
\eqref{stratified equation example} is not stratified for any choice of $a_1,a_0.$ In this case, by Theorem~\ref{intro2}  $\level (y^2z^3-x^5-bz^5)=2.$
\end{ex}

\begin{ex}
Given $p>17,$ consider the equation
\begin{equation}\label{stratified equation example 2}
(x')^2=\frac{(x-1)^8-x^8}{(a_2x^2+a_1x+a_0)^2}.
\end{equation}
One can check that, under a M\"obius transformation of the form
\[
(x,y)\longmapsto\left(\frac{1}{x+1},\frac{y}{(x+1)^4}\right),
\]
the hyperelliptic curve $\mathcal{H}$ defined by $y^2=(x-1)^8-x^8$ corresponds to $y^2=x^8-1,$ and therefore both have the same $p$--rank. As shown in
\cite[Section 2]{KodamaTopWashio2009}, $\mathcal{H}$ is ordinary if and only if $p\equiv 1\pmod 8,$ and supersingular (that is, its $p$--rank is $0$) if and only if $p\equiv 7\pmod 8.$ In the ordinary case, we know that the level is $2,$ and at least three in the supersingular (not superspecial) case. However, in the remaining cases (where $p\equiv 3,5\pmod 8$) the curve has $p$--ranks $1$ and $2$ respectively, and in these two cases, while we can ensure that there are non--zero choices of $a_2,a_1,a_0$ such that \eqref{stratified equation example 2} can be either stratified or not, we can not predict in general what is the level.
\end{ex}

\section{The level of a pair of polynomials}\label{the level of a
pair: basic properties}
Hereafter, let $k$ be a perfect field
of prime characteristic $p$, and let
$R$ be the polynomial ring $k[x_1,\ldots, x_d]$. The aim of this section is to study the following concept.

\begin{df}\label{level of a pair}
Given polynomials $f,g$ with coefficients in $k$ and $f \neq 0$, one defines
the \emph{level of $(g,f)$} as 
\[
\level (g,f):=\inf\{e\geq 0:\ \exists\delta\in \mathcal D^{(e)}\text{ such that }\delta(g/f)=(g/f)^p\} \in \mathbb{N}_0 \cup \{\infty\}.
\]
When $g=1$, one denotes $\level(f)$ instead of
$\level (1,f)$; this is the
notion of level of a polynomial introduced in
\cite[Definition 2.6]{BoixDeStefaniVanzo2015}. 
\end{df}
\begin{rk}
Note that $\level(g,f)$ only depends on the quotient $g/f$, so one could also reasonably denote this notion by $\level(\frac{g}{f})$ instead. But this alternative notation is inconsistent with the one in \cite{BoixDeStefaniVanzo2015} in the case $f = 1$, so we stick with the notation $\level(g,f)$. In any case, one can usually assume that $g$ and $f$ are coprime, since common factors do not change the level of the pair. 

Note also that $\level(g,f) = 0$ if and only if $g/f \in R$. If $g$ and $f$ are coprime, this only happens if $f$ is a constant. 

In Proposition \ref{maybe of infinite level} we give an example of polynomials $f$ and $g$ such that $\level(g,f) = \infty$. 
\end{rk}

Before going on studying this notion, we review
the so--called ideals of $p^e$th roots; the interested reader
can find a more detailed treatment in \cite[page 465]{AlvarezBlickleLyubeznik2005}, \cite[Definition 2.2]{BlickleMustataSmith2008} and
\cite[Definition 5.1]{Katzman2008}. For an ideal $I \subset R$ we denote by $I^{[p^e]}$ the ideal generated by the $p^e$-th powers of elements of $I$. 

\begin{df}\label{ideals of pe roots} 
Given $g\in R$ and an integer $e \geq 0$, we define the \emph{ideal of $p^e$th roots} $I_e (g)$ to be the smallest ideal $J\subseteq R$ such that $g\in J^{[p^e]}.$
\end{df}
\begin{rk}
Under our assumptions, $R$ is a free $R^{p^e}$-module with basis given by the monomials $\{\mathbf{x}^{\mathbf{\alpha}} \mid \norm{\alpha} \leq p^e-1\}$. A polynomial $g \in R$ can therefore be written as 
\[
g=\sum_{0\leq\norm{\mathbf{\alpha}}\leq p^{e}-1}g_{\mathbf{\alpha}}^{p^e} \mathbf{x}^{\alpha},
\]
for unique $g_\alpha \in R$. Then $I_e (g)$ is the ideal of $R$ generated by elements $g_{\mathbf{\alpha}}$ \cite[Proposition 2.5]{BlickleMustataSmith2008}. 
\label{remarkI}
\end{rk}

The main relation between these ideals and differential operators is the following equality, valid for any polynomial $g \in R$ and any integer $e \geq 0$ (see \cite[Lemma 3.1]{AlvarezBlickleLyubeznik2005}): 
\begin{equation}\label{pe-th roots and differential operators}
  \mathcal{D}^{(e)}\cdot g = I_{e}(g)^{[p^e]}.  
\end{equation}
Using this, one can relate the level of a pair of polynomials to ideals of $p^e$th roots as follows. 

\begin{lm}\label{p root ideal entering} 
Let $f, g \in R$ and $e \geq 0$ be given. Then the following are equivalent:
\begin{enumerate}[(i)]
    \item  $\level(g,f) \leq e;$
    
    \item $I_e (g^pf^{p^e-p})\subseteq
I_e (gf^{p^e-1});$

    \item $I_e (g^pf^{p^e-p})^{[p^e]} \subseteq I_e (gf^{p^e-1})^{[p^e]}.$
\end{enumerate}
In particular, $\level (g,f) =\inf\{e\geq 0:\ I_e (g^pf^{p^e-p})\subseteq
I_e (gf^{p^e-1})\}.$
\end{lm}

\begin{proof}
The equivalence of (ii) and (iii) is proved in the last paragraph of the proof of \cite[Proposition 3.5]{AlvarezBlickleLyubeznik2005}. We prove that (i) and (iii) are equivalent. Suppose that there is $\delta\in \mathcal D^{(e)}$ such that $\delta(g/f)=(g/f)^p$. Since $\delta$ is linear over $p^e$-powers, this
implies that $\delta (gf^{p^e-1})=g^p f^{p^e-p}$. By \eqref{pe-th roots and differential operators}, this implies $g^pf^{p^e-p} \in I_e(g^pf^{p^e-p})^{[p^e]}$, so that $I_e (g^pf^{p^e-p})^{[p^e]}\subseteq
I_e (gf^{p^e-1})^{[p^e]}$.

Conversely, suppose now that $I_e (g^pf^{p^e-p})^{[p^e]}\subseteq
I_e (gf^{p^e-1})^{[p^e]}$. Again using
\eqref{pe-th roots and differential operators}, one has that
$\mathcal D^{(e)} (g^p f^{p^e-p})\subseteq \mathcal D^{(e)} (gf^{p^e-1})$. In particular
$g^p f^{p^e-p}\in \mathcal  D^{(e)} (gf^{p^e-1})$, hence
there is $\delta\in D^{(e)}$ such that
$\delta (gf^{p^e-1})=g^p f^{p^e-p}$. Multiplying
this equality by $1/f^{p^e}$ and using that $\delta$ is linear over $p^e$th powers, we get $\delta(g/f)=(g/f)^p$. 
\end{proof}

Observe that the equality
$\mathcal{D}^{(e)}\cdot g = I_{e}(g)^{[p^e]}$
is made explicit in, e.g., the proof of
\cite[Claim~3.4]{BoixDeStefaniVanzo2015}.
Using these techniques one can in case
$e=\level(g,f)<\infty$, algorithmically
construct an explicit operator 
$\delta\in\mathcal{D}_R^{(e)}$
with $\delta(g/f)=g^p/f^p$. However we do
not know how to decide whether the level
of a given pair is finite.

Our next goal is to show that the level of a pair is invariant under coordinate transformations. Although \ref{equivalent fractions under change of coordinates}, 
\ref{homogeneous change of coordinates give p basis} and
\ref{behavior of p roots linear change} can also be found in
\cite{BlancoBoixFordhamYilmaz2018}, we review it here for the convenience
of the reader.

Denote $G:=\operatorname{GL}_d (k)$ and observe that $R$ has a right action of $G$ defined by $(f|A)(x_1,...,x_d):=f(y_1,...,y_d)$, where
\[
\begin{pmatrix} y_1\\ \vdots\\ y_d\end{pmatrix}=A\cdot\begin{pmatrix} x_1\\ \vdots\\ x_d\end{pmatrix},
\]
for $A\in G$.
Observe as well that a matrix $A\in G$ induces an isomorphism $\phi_A$ of $k$--algebras $\xymatrix@1{R\ar[r]^-{\phi_A}& R}$ defined by $\phi_A(f)=f|A$, the inverse being given by $\phi_{A^{-1}}$. 
\begin{df}\label{equivalent fractions under change of coordinates}
Given homogeneous $f,g\in R$, we say that $f$ and $g$ are \emph{$G$--equivalent} if there is $A\in G$ such that $\phi_A(f)=g$.
\end{df}
We need the following easy fact.

\begin{lm}\label{homogeneous change of coordinates give p basis}
Notations as before, let $y_1,\ldots ,y_d\in R$ be homogeneous elements of degree $1$ such that
\[
\begin{pmatrix} y_1\\ \vdots\\ y_d\end{pmatrix}=A\cdot\begin{pmatrix} x_1\\ \vdots\\ x_d\end{pmatrix},
\]
for some $A\in G$. Then, for any $e\geq 0$ the set
\[
\mathcal{B}:=\{\mathbf{y}^{\alpha}:=y_1^{a_1}\cdots y_d^{a_d}:\ \alpha =(a_1,\ldots ,a_d)\in\N^d,\ 0\leq a_i\leq p^e-1\text{ for any }1\leq i\leq d\}
\]
is a basis of $R$ as $R^{p^e}$--module.
\end{lm}

\begin{proof}
We have $\mathbf{y}^\alpha = \phi_A(\mathbf{x}^\alpha)$ for every multi-index $\alpha$. Therefore, the set $\mathcal B$ is the image of the $R^{p^e}$-basis $\{\mathbf{x}^\alpha: \|\alpha\| \leq p^e-1\}$ under the $k$-algebra isomorphism $\phi_A$. Since $\phi_A$ restricts to an isomorphism on $R^{p^e}$, the result follows.  
\end{proof}

In this way, we are ready to prove the following:
\begin{teo}\label{behavior of p roots linear change}
For any $f\in R$, $A\in G$ and $e\geq 0$, it holds that
$\phi_A (I_e (f))=I_e(\phi_A(f))$. In particular, for all $f, g \in R$ we have $\level(f,g) = \level(\phi_A(f), \phi_A(g))$. 
\end{teo}

\begin{proof}
Setting
\[
\begin{pmatrix} y_1\\ \vdots\\ y_d\end{pmatrix}=A\cdot\begin{pmatrix} x_1\\ \vdots\\ x_d\end{pmatrix},
\]
and applying Lemma \ref{homogeneous change of coordinates give p basis} we see that the set
\[
\{\mathbf{y}^{\alpha}:=y_1^{a_1}\cdots y_d^{a_d}:\ \alpha =(a_1,\ldots ,a_d)\in\N^d,\ 0\leq a_i\leq p^e-1\text{ for any }1\leq i\leq d\}
\]
is a basis of $R$ as $R^{p^e}$--module.
Write
\[
f=\sum_{0\leq\norm{\alpha}\leq p^e-1}f_{\alpha}^{p^e}\mathbf{x}^{\alpha},
\]
for $f_{\alpha}\in R$. Then
\[
\phi_A(f)=\sum_{0\leq\norm{\alpha}\leq p^e-1}\phi_A(f_{\alpha})^{p^e}\mathbf{y}^{\alpha},
\]
which shows that $\phi_A (I_e (f))\subseteq I_e (\phi_A(f))$. Equality holds because 
$\phi_A$ is an isomorphism. The second claim follows from the first together with Lemma \ref{p root ideal entering}.
\end{proof}
In the next statement, our aim is to collect some properties that
the level of a pair of polynomials satisfy.

\begin{prop}\label{level of a pair of polynomials: properties}
Let $f, g \in R$ be non-zero polynomials such that $\frac{g}{f} \notin R$. Then the following statements hold.
\begin{enumerate}[(i)]
\item $\level (g,f)=1$ if and only if $g\in I_1 (gf^{p-1}).$
\item If $\level (f)=1$, then $\level (g,f)=1.$
\item If either $I_e (g^pf^{p^e-p})\not\subseteq I_e(f^{p^e-1})$
or $I_e (g^pf^{p^e-p})\not\subseteq I_e (g)$, then
$\level(g,f) > e.$
\item If $f$ and $g$ are homogeneous, and $e\geq 1$ is an integer such that $p^e > \deg g - \deg f$, then
$I_e (gf^{p^e-1})$ is generated by polynomials of degree at
most $\deg f.$
\end{enumerate}
\end{prop}
\begin{proof}
The assumption that $f$ does not divide $g$ in $R$ implies that $\level(g,f) > 0$. Then (i) follows from Lemma \ref{p root ideal entering} together with the easy observation that $I_1(g^p) = (g)$. Part (ii) was already proved in \cite[page 248]{BoixDeStefaniVanzo2015}; we
repeat the proof for the sake of completeness. Let
$\delta'\in \mathcal D^{(1)}$ such that $\delta ' (1/f)=1/f^p$. Then define 
$\delta:=\delta '\circ (\cdot g^{p-1})$. We find that
$\delta (g/f)=\delta ' (g^p/f)=g^p \delta ' (1/f)=(g/f)^p$, as
required.

Part (iii) follows immediately combining
Lemma \ref{p root ideal entering} with the fact that
$I_e (g f^{p^e-1})\subseteq I_e(g) I_e (f^{p^e-1})$
\cite[Lemma 3.3]{AlvarezBlickleLyubeznik2005}. Finally, to
prove part (iv) fix $e\geq 1$ an integer and write
\[
gf^{p^e-1}=\sum_{0\leq\norm{\alpha}\leq p^e-1}c_{\alpha}^{p^e}\mathbf{x}^{\alpha},
\]
for some $c_{\alpha}\in R$. Since both $f$ and $g$ are homogeneous
it follows that
\[
\deg(g)+(p^e-1)\deg (f)=p^e \deg (c_{\alpha})+\deg (\mathbf{x}^{\alpha}),
\]
which implies that
\[
\deg (c_{\alpha})\leq\frac{(p^e-1)\deg (f)+\deg (g)}{p^e}= \deg (f)+ \frac{\deg g - \deg f}{p^e}.
\]
The second term on the right hand side is smaller than 1 by assumption, and since both sides are integers, we get $\deg c_\alpha \leq \deg f$. The result follows. 
\end{proof}

\section{Some examples}\label{some concrete calculations
of the level of a pair}
The goal of this section is to calculate the level of a pair of polynomials $(g,f)$ for several particular choices of $g$ and $f$; we will quickly see that, even for low degrees, most of the calculations are highly non--trivial. In particular, we show that $\level(g,f)$ is, in general, not always finite (see Example \ref{maybe of infinite level}).

We want to start with the case considered by Singh, 
see for example \cite{Singh2017}.

\begin{lm}\label{Singh determinantal example}
Let $p$ be a prime number, $X=\begin{pmatrix}
u& v& w\\ x& y& z\end{pmatrix}$ be a matrix of indeterminates
defined over $R=k [u,v,w,x,y,z]$, and set $\Delta_1:=vz-wy$, $\Delta_2:=wx-uz$ and
$\Delta_3:=uy-vx$. Then, $\level (g,f)=1$ for each pair
$(g,f)\in\left\{(w,\Delta_1\Delta_2),\, (v,\Delta_1\Delta_3),\,(u,\Delta_2\Delta_3) \right\}.$
\end{lm}

\begin{proof}
By symmetry, it is enough to show that $\level (g,f)=1$ when $(g,f)=(w,\Delta_1\Delta_2)$.
Set $f:=\Delta_1\Delta_2$, and notice that $f=1\cdot (xzvw)+ (-z^2)\cdot (uv)+(-w^2)\cdot (xy)+1\cdot (yzuw)$. This shows that, if $p=2$, then $I_1 (f)=R$ so $\level (f)=1$ and therefore $\level(g,f)=1$. Now, assume that $p\geq 3$, one can check that in the support of $f^{p-1}$ appears the monomial $(xyuv)^{(p-1)/2}(zw)^{p-1}$ with coefficient $\binom{p-1}{(p-1)/2}$; this shows again that $\level (f)=1$ and therefore $\level(g,f)=1.$
\end{proof}

\begin{rk}
Notice that, in the setting considered in Lemma \ref{Singh determinantal example}, Singh shows in \cite{Singh2017} that the differential operator
$\delta:=D_{u,p-1}D_{y,p-1}D_{z,p-1}$ (which is clearly of level one) is such that $\delta(g/f)=(g/f)^p$, for $g/f$ any of the three fractions considered in Lemma \ref{Singh determinantal example}.
\end{rk}

\begin{lm}\label{the case of the projective line}
Let $k$ be a field of characteristic $p,$ let $f=x^d,$ assume that $p\geq d,$ and let
$g\in R=k[x,y]$ be a homogeneous polynomial of degree
$d$ which is not a multiple of $f$. Then, $\level (g,f)=2$ unless $g\in (x^{d-1})$, in which
case $\level (g,f)=1.$
\end{lm}

\begin{proof}
Write $g=\sum_{i=0}^d a_{i}x^i y^{d-i}$; now, notice that
\[
gf^{p-1}=\sum_{i=0}^d a_{i}x^{i+d(p-1)} y^{d-i}.
\]
Given $0\leq i\leq d$ write $i+d(p-1)=(d-1)p+(p+i-d)$, and notice that, unless
$i=d$, $0\leq p+i-d\leq p-1$ (here, we are also using that
$d\leq p$). This shows that
$I_1 (gf^{p-1})= (a_{d}^{1/p}x^d, a_{i}^{1/p}x^{d-1}:\ 
1\leq i\leq d-1)= (x^{d-1}),$
so $\level (g,f)\neq 1$ unless $g\in (x^{d-1})$, in which case $\level (g,f)=1$. So, from
now on, assume that $g\notin (x^{d-1})$.

We have $I_2(g^pf^{p^2-p}) = I_1(gf^{p-1}) = (x^{d-1})$. 
Now, write
\[
gf^{p^2-1}=\sum_{i=0}^d a_{i}x^{i+d(p^2-1)} y^{d-i}.
\]
Again, the equality $i+d(p^2-1)=(d-1)p^2+(p^2+i-d)$ and the fact
unless $i=d$, $p^2+i-d\leq p^2-1$, shows that 
$I_2 (gf^{p^2-1}) = (a_{d}^{1/p^2}x^d, a_{i}^{1/p^2}x^{d-1}:\ 
1\leq i\leq d-1)= (x^{d-1}) = I_2(g^pf^{p^2-p})$,
and therefore $\level (g,f)=2$, as claimed.
\end{proof}

Lemma \ref{the case of the projective line} has the following interesting
consequence.

\begin{lm}\label{the case of quadratic forms}
Let $k$ be a field of prime characteristic $p,$ and let $f,g\in k[x,y]$ be quadratic forms. If
$\sqrt{(f)}$ denotes the radical of $(f)$, then
\[
\level (g,f)=\begin{cases} 0, \text{ if $g$ is a multiple of $f$}, \\
1,\text{ if either $f$ is not the square of a linear form, or if }g\in\sqrt{(f)},\\
2,\text{ otherwise.}\end{cases}
\] 
\end{lm}

\begin{proof}
First of all, if $f$ is not the square of a linear form, then by
\cite[Proposition 5.7]{BoixDeStefaniVanzo2015} $\level (f)=1$ and
therefore part (ii) of Proposition \ref{level of a pair of polynomials: properties}
implies that $\level(g,f)=1$. So, hereafter we assume that
$f$ is the square of a linear form; thanks to
Theorem \ref{behavior of p roots linear change} we
can assume, without loss of generality, that
$f=x^2$ and that $g$ is again a quadratic form. Then, in this
case, Lemma \ref{the case of the projective line} says exactly
that $\level(g,f)=2$ unless $g\in (x)$, in which case
$\level(g,f)=1$; the proof is therefore completed.
\end{proof}

As a more elaborate example we now consider
$\level(g,f)$ with $f=x^3+y^3+z^3$ and $g$
any homogeneous cubic in $3$ variables which is not a scalar multiple of $f$.
Since $\level(f)=1$ in case the characteristic
$p\equiv 1\pmod 3$, Proposition~\ref{level of a pair of polynomials: properties}~(ii)
shows $\level(g,f)=1$ for $p\equiv 1\pmod 3$
and any such $g$.

We expect that the same holds for all 
characteristics $p\geq 5$. The next two special
cases show that this is correct for
most $g$. By Example~\ref{p can not be two in the Fermat case}, the same
does not hold in characteristics $p=2,3$.

\begin{cl}\label{the Fermat cubic case}
Let $p\geq 5$ with $p\equiv 2 \pmod 3$, let $f=x^3+y^3+z^3$, and let
$g\in R=k[x,y,z]$ be a homogeneous polynomial
of degree $3$ such that, if one writes $g=\sum_{a+b+c=3}g_{a,b,c}x^a y^b z^c$, and set $B:=\binom{p-1}{(p-2)/3,(p-2)/3,(p+1)/3}$, 
$C:=\binom{p-1}{(p-2)/3}$, $D:=\binom{p-1}{1,2(p-2)/3,(p-2)/3},$
$E:=\binom{p-1}{1,(2p-1)/3,(p-5)/3}$, and
$F:=\binom{p-1}{(p+4)/3,(p-2)/3,(p-5)/3}$, then the rank of
\[
A:=\begin{pmatrix} Bg_{1,1,1}& Cg_{2,0,1}& Cg_{2,1,0}\\
Cg_{0,2,1}& Bg_{1,1,1}& Cg_{1,2,0}\\
Cg_{0,1,2}& Cg_{1,0,2}& Bg_{1,1,1}\\
B g_{2,0,1}& Dg_{0,2,1}& Cg_{3,0,0}+Eg_{0,3,0}+Dg_{0,0,3}\\
B g_{2,1,0}& Cg_{3,0,0}+Eg_{0,3,0}+Dg_{0,0,3}& Dg_{0,1,2}\\
B g_{3,0,0}+Fg_{0,3,0}+Fg_{0,0,3}& Dg_{1,2,0}& Dg_{1,0,2}\\
D g_{2,0,1}& Bg_{0,2,1}& Eg_{3,0,0}+Cg_{0,3,0}+Dg_{0,0,3}\\
Eg_{3,0,0}+Cg_{0,3,0}+Dg_{0,0,3}& Bg_{1,2,0}& Dg_{1,0,2}\\
D g_{2,1,0}& F g_{3,0,0}+Bg_{0,3,0}+Fg_{0,0,3}& Dg_{0,1,2}\\
D g_{2,1,0}& Eg_{3,0,0}+Dg_{0,3,0}+Cg_{0,0,3}& Bg_{0,1,2}\\
E g_{3,0,0}+D g_{0,3,0}+C g_{0,0,3}& Dg_{1,2,0}& Bg_{1,0,2}\\
D g_{2,0,1}& Dg_{0,2,1}& F g_{3,0,0}+F g_{0,3,0}+ Bg_{0,0,3}\end{pmatrix}
\]
is three. Then $\level (g,f)\leq 1$, with equality exactly if $g$ is not a multiple of $f$.
\end{cl}

\begin{proof}
Write $g=\sum_{a+b+c=3}g_{a,b,c}x^a y^b z^c$, and
\[
gf^{p-1}=\sum_{a+b+c=3}\sum_{i+j+k=p-1}
g_{a,b,c}\binom{p-1}{i,j,k}x^{3i+a}y^{3j+b}
z^{3k+c}.
\]
Then, if one picks $i=j=(p-2)/3$ and $k=(p+1)/3$, then
the corresponding term of $gf^{p-1}$ is
\[
\sum_{a+b+c=3}g_{a,b,c}\binom{p-1}{i,j,k}
z^p \cdot (x^{p-2+a}y^{p-2+b}z^{c+1}).
\]
Again, if $i=k=(p-2)/3$ and $j=(p+1)/3$, then
the corresponding term of $gf^{p-1}$ is
\[
\sum_{a+b+c=3}g_{a,b,c}\binom{p-1}{i,j,k}
y^p \cdot (x^{p-2+a}y^{b+1}z^{p-2+c}).
\]
By the same reason, if $j=k=(p-2)/3$ and $i=(p+1)/3$, then
the corresponding term of $gf^{p-1}$ is
\[
\sum_{a+b+c=3}g_{a,b,c}\binom{p-1}{i,j,k}
x^p \cdot (x^{a+1}y^{p-2+b}z^{p-2+c}).
\]
The above expansions show that the
basis elements $x^{p-1}y^{p-1}z^2$, $x^{p-1}y^2z^{p-1}$ and
$x^2y^{p-1}z^{p-1}$ contain respectively in their coefficient the below term, where $B:=\binom{p-1}{(p-2)/3,(p-2)/3,(p+1)/3}$:
\[
g_{1,1,1}B z^p,
g_{1,1,1}B y^p,
g_{1,1,1}B x^p.
\]
Hereafter, we only plan to prove that the coefficient of
$x^{p-1}y^{p-1}z^2$ is exactly $Cg_{0,1,2}x^p+
Cg_{1,0,2}y^p+Bg_{1,1,1} z^p$
and one can show using the same arguments that the
coefficient of $x^{p-1}y^2z^{p-1}$ (resp. $x^2y^{p-1}z^{p-1}$)
is exactly $Cg_{0,2,1}x^p+Bg_{1,1,1}y^p+Cg_{1,2,0} z^p$ resp. 
$Bg_{1,1,1}x^p+Cg_{2,0,1}y^p+Cg_{2,1,0} z^p.$

Indeed, we want to calculate the coefficient of $x^{p-1}y^{p-1}z^2$, so
suppose that there are non--negative integers $\lambda,\mu,\gamma$
such that $3i+a=\lambda p+p-1,\ 3j+b=\mu p+p-1,\ 3k+c=\gamma p+2.$
Since $\deg (gf^{p-1})=3p$, it follows that $3p=3i+a+3j+b+3k+c=(\lambda+\mu+\gamma +2)p,$
which implies that $\lambda+\mu+\gamma=1$, so we only have
three possibilities for these integers; namely, $(1,0,0)$, $(0,1,0)$ and
$(0,0,1)$. For $(1,0,0)$, we get $i=(2p-1-a)/3,\ j=(p-1-b)/3,\ k=(2-c)/3.$
Since $p\equiv 2\pmod 3$, this forces $a=0$, $b=1$ and $c=2$.  By the
same argument, for $(0,1,0)$ one gets $a=1$, $b=0$ and $c=2$, and
finally, for $(0,0,1)$ one ends up with
$a=b=c=1$. This shows that the coefficient of
$x^{p-1}y^{p-1}z^2$ is exactly $B(g_{0,1,2}x^p+
g_{1,0,2}y^p+g_{1,1,1} z^p)$, as claimed.

One might ask from where the other rows of matrix $A$ appearing in
our assumption comes from; following the same arguments, these rows
corresponds to the calculation of the coefficients of the below
basis elements:
\begin{align*}
& x^3y^{p-2}z^{p-1},\ x^3y^{p-1}z^{p-2},\ x^4y^{p-2}z^{p-2},\\
& x^{p-2}y^3z^{p-1},\ x^{p-1}y^3z^{p-2},\ x^{p-2}y^4z^{p-2},\\
& x^{p-2}y^{p-1}z^3,\ x^{p-1}y^{p-2}z^3,\ x^{p-2}y^{p-2}z^4.
\end{align*}

Summing up, the foregoing implies, since by assumption
the rank of $A$ is $3$, that $(x,y,z)=I_1 (gf^{p-1})$, hence
$g\in I_1 (gf^{p-1})$ and this shows that $\level (g,f)=1$ by using
part (i) of Proposition \ref{level of a pair of polynomials: properties}.
\end{proof}

\begin{cl}\label{homogeneous fractions of degree 3}
Let $p\geq 5$, let $f=x^3+y^3+z^3$, and let
$g\in R=k[x,y,z]$ be a non--zero monomial
of degree $3$. Then, $\level(g,f)=1.$
\end{cl}
\begin{proof}
If $p\equiv 1 \pmod 3$, then $\level (f)=1$ and therefore $\level(g,f)=1$
by part (ii) of Proposition \ref{level of a pair of polynomials: properties}, so hereafter we will assume that $p\equiv 2\pmod 3$. By symmetry, it is enough to consider the monomials $g = x^3$, $g = x^2y$ and $g = xyz$. In each of these cases, we will simply construct an explicit differential operator of level $1$ that does what is needed. For $g = x^3$, consider first
\[ \delta = D_{x,p-1} \circ D_{y, p-2} \circ D_{z, 3} \]
(see the Introduction for the notation $D_{x,n}$). Clearly $\delta$ is of level 1, since $p > 3$. 
We have that 
\[ gf^{p-1} = \sum_{i+j+k = p-1} \binom{p-1}{i,j,k} x^{3i+3}y^{3j}z^{3k}.\]
Applying $\delta$ gives us
\[ \delta(gf^{p-1}) = \sum_{i+j+k = p-1} \binom{p-1}{i,j,k} \binom{3i+3}{p-1}\binom{3j}{p-2}\binom{3k}{3} x^{3i+4-p}y^{3j+2-p}z^{3k-3},\]
where we use the convention that $\binom{n}{k} = 0$ for $k > n$. We investigate for which indices $i, j, k$ the coefficient in this term is zero. The first factor is never zero, since $p-1$, $i$, $j$ and $k$ are all between $0$ and $p-1$. The second factor is zero unless $3i + 3 \equiv -1 \pmod p$, as can be seen by writing out the product. Since $i$ lies between $0$ and $p-1$, and since $p \equiv 2 \pmod 3$, the only integer value for $i$ such that $3i + 3 \equiv -1 \bmod p$ is $i = (2p-4)/3$. This means that $j$ is at most $(p+1)/3$. The third factor $\binom{3j}{p-2}$ is zero unless $3j$ is either $-1$ or $-2$ modulo $p$. In the allowed range for $j$, the only integer possibility is $j = (p-2)/3$. This leaves $k = 1$, and for this value of $k$ we have $\binom{3k}{3} = 1 \neq 0$. So we see that the only non-zero term in $\delta(gf^{p-1})$ is the one for indices $(i,j,k) = (\frac{2p-4}{3}, \frac{p-2}{3}, 1)$. This gives
\[ \delta(gf^{p-1}) = \binom{p-1}{\frac{2p-4}{3}, \frac{p-2}{3}, 1} \binom{2p-1}{p-1}\binom{p-2}{p-2}\binom{3}{3}x^{p} = \binom{p-1}{\frac{2p-4}{3}, \frac{p-2}{3}, 1} x^{p}\]
Define now
\[ \Delta = \binom{p-1}{\frac{2p-4}{3}, \frac{p-2}{3}, 1}^{-1} \cdot x^{2p}\cdot \delta,\]
then $\Delta$ is also a differential operator of level 1, and by construction we have $\Delta(gf^{p-1}) = x^{3p} = g^p$. Using that $\Delta$ is $R^p$-linear, we may divide both sides by $f^p$ and get $\Delta(g/f) = g^p/f^p,$ as needed.

For the other cases $g=x^2y$ and $g = xyz$, a similar analysis shows that the operators
\[ C \cdot x^{p}y^p D_{x,p-2}D_{y,p-1}D_{z,3}, \quad \textrm{resp.} \quad C' \cdot y^pz^p, D_{x,p-3}D_{y,p-1}D_{z,4}\]
for suitably chosen non-zero constants $C, C'  \in \mathbb F_p,$ have the required property.
\end{proof}

Notice that, in the example considered in
Lemma \ref{the Fermat cubic case}, $\level (f)=2>\level (g,f)=1$. From here, one
might ask whether, in general, $\level (g,f)\leq\level (f)$; however, this
is not the case, as the below example shows. The unjustified
calculations were done with Magma \cite{Magma}.

\begin{ex}\label{level of pair not bounded by level of denominator}
Let $R=k [x,y,z,w]$, $g=y$ and $f=xy^{p+1}
+yz^{p+1}+zw^{p+1}$; when $p\in\{2,3,5\}$, $\level(g)=1,$
$\level(f)=2$, but $\level(g,f)=4.$

For any prime $p$, what is easy to show in this example is that $\level(g,f)\geq 2$; indeed, notice
that
\[
gf^{p-1}=\sum_{\substack{0\leq i,j,k\leq p-1\\ i+j+k=p-1}}
\binom{p-1}{i,j,k}(y^iz^jw^k)^p \cdot (x^i y^{p-k}z^{p-1-i}w^k).
\]
We claim that, whereas $y^p\in I_1 (gf^{p-1})$, $g=y\notin I_1 (gf^{p-1})$. Indeed, if in the
above expansion we pick $j=k=0$ and $i=p-1$, then one gets that
$gf^{p-1}=(y^p)^p (x^{p-1})+\ldots$, and this choice is the only one
that makes the basis element $x^{p-1}$ appearing in this expansion. This
shows that $y^p\in I_1 (gf^{p-1})$; moreover, notice that, if one
choices a $i,j,k$ as above where $i<p-1$, then the coefficient
of the corresponding basis element is made up by monomials that are divisible
by either $z$ or $w$. This shows that $y^p$ is the smallest possible
power of $y$ that belongs to $I_1 (gf^{p-1})$, hence
$g=y\notin I_1 (gf^{p-1})$ and therefore $\level (g,f)\geq 2$, as
claimed.
\end{ex}
Moreover, again about Lemma \ref{homogeneous fractions of degree 3}, we want to
single out that the assumption $p\neq 2,3$ can not be removed, as the following examples show.

\begin{ex}\label{p can not be two in the Fermat case}
Let $p=2$, let $R=k[x,y,z]$, $f=x^3+y^3+z^3$ and
$g=xyz$; we claim $\level (g,f)=2$. Indeed, on the one hand,
$gf^{p-1}=(x^2)^2\cdot (yz)+(y^2)^2\cdot (xz)+(z^2)^2\cdot (xy)$, so $g=xyz\notin I_1(gf^{p-1})=(x^2,y^2,z^2)$; this shows, by part (i)
of Proposition \ref{level of a pair of polynomials: properties}, that
$\level (g,f)\geq 2$. On the other hand,
\begin{align*}
gf^{p^2-1}= & (x^2)^4\cdot (x^2yz)+(y^2)^4\cdot (xy^2z)
+(z^2)^4\cdot (xyz^2)+(xy)^4 \cdot (x^3z)+
(xy)^4 \cdot (y^3z)\\ &+(xz)^4 \cdot (x^3y)+
(xz)^4 \cdot (yz^3) +(yz)^4 \cdot (xyz^3),
\end{align*}
and $g^p f^{p^2-p}=x^8 (yz)^2+y^8 (xz)^2+z^8 (xy)^2$; these
last two computations show that
\[
g^p f^{p^2-p}\in (x^2,y^2,z^2,xy,xz,yz)^{[p^2]}=I_2 (gf^{p^2-1})^{[p^2]},
\]
and therefore Lemma \ref{p root ideal entering} ensures
$\level(g,f)=2$, as claimed.

Now, assume that $p=3$ ($g$ and $f$ are the same); in this case, one can check that $J:=I_1 (gf^2)=(x^2+2xy+y^2+2xz+2yz+z^2)$ and $g=xyz\notin J.$ One way to check it is the following; denote
by $V(J)$ the hypersurface defined by $J.$ This hypersurface contains the point $(1,1,1),$ which is a point which does not belong to $V(xyz).$ This shows that $xyz\notin J.$ 

The above argument shows that $\level (g,f)\geq 2$ and, actually, one can check either by hand or by computer that $\level (g,f)=2.$
\end{ex}

We conclude this section with an example showing that the
level of a pair of polynomials is, in general, not finite. This in fact answers a question raised in \cite[Section 5]{BoixDeStefaniVanzo2015}.

\begin{prop}\label{maybe of infinite level}
Let $R=k[x,y]$ with $\operatorname{char} k = p$, and let $f=x^{p+1}+y^{p+1}$ and
$g=x$.
Then $\level(g,f)=\infty$.
In particular, no $\delta\in \mathcal{D}_R$
exists with $\delta(g/f)=g^p/f^p$.
\end{prop}
\begin{proof}
Let $e \geq 2$ be an arbitrary even integer. We will show that $\level(g,f) > e$. By Lemma \ref{p root ideal entering}, this is equivalent to showing that $I_e(g^pf^{p^e-p})^{[p^e]} \not \subseteq I_e(gf^{p^e-1})^{[p^e]}$.

First we show that $I_e(gf^{p^e-1})$ is a monomial ideal. Indeed, we have 
\begin{equation} \label{eq:gfpe-1}
gf^{p^e-1} = \sum_{i = 0}^{p^e-1} \binom{p^e-1}{i} x^{i(p+1)+1}y^{(p^e-1-i)(p+1)}.
\end{equation}
By the description of $I_e$ in Remark \ref{remarkI}, to find generators of $I_e(gf^{p^e-1})$, we take out $p^e$-th powers. If for two indices $i$ and $j$ the corresponding terms above differ by a $p^e$-th power, then they both contribute to the same generator. But this happens only if the exponents for $x$ and $y$ are congruent modulo $p^e$. From $i(p+1) + 1 \equiv j(p+1) + 1 \pmod{p^e}$ we obtain $i \equiv j \pmod{p^e}$ since $p+1$ is a unit modulo $p^e$. But if $0 \leq i,j \leq p^e - 1$ and $i \equiv j \pmod{p^e}$, then $i = j$. So we see that the terms occurring in $gf^{p^e-1}$ are independent over $\operatorname{Frac}(R^{p^e})$. Hence the generators for $I_e(gf^{p^e-1})$ that we get from Remark \ref{remarkI} are monomials, and so $I_e(gf^{p^e-1})$ is a monomial ideal. It follows that also $I_e(gf^{p^e-1})^{[p^e]}$ is a monomial ideal. 

Now we show that $g^pf^{p^e-p} \notin I_e(gf^{p^e-1})^{[p^e]}$. Since the latter is a monomial ideal, it is sufficient to find a monomial that occurs in $g^pf^{p^e-p}$ with non-zero coefficient which is not in this ideal. For this, set $m := x^{p^{e} - p^2 + p}y^{p^{e+1} - p}.$
We claim that this monomial occurs in $g^pf^{p^e-p}$ with non-zero coefficient. We have
\[ g^pf^{p^e-p} = \sum_{i = 0}^{p^e-p} \binom{p^e-p}{i} x^{i(p+1)+p}y^{(p^e-p-i)(p+1)}.\]
We see that our monomial $m$ occurs for index $i = (p^e - p^2)/(p+1)$, which is an integer because $e$ is even. To evaluate the binomial coefficient for this value of $i$, we can look at the p-adic digits of the numbers involved. We have $p^e - p = (p-1)p^{e-1} + (p-1)p^{e-2} + \ldots + (p-1)p$, and we have $i = (p-1)p^{e-2} + (p-1)p^{e-4} + \ldots + (p-1)p^2$. Using Lucas's theorem \cite[pp. 51--52]{Lucas}, we find that the binomial coefficient evaluates to $1$, so in particular it is non-zero. 

Now we need to show that $m \notin I_e(gf^{p^e-1})^{[p^e]}$. This ideal is generated by monomials which are also $p^e$-th powers, and $m$ is an element of this ideal if and only if at least one of these monomials divides $m$. The largest $p^e$-th power dividing $m$ is $y^{(p-1)p^e}$. Hence, it is enough to show that $y^{(p-1)p^e} \notin I_e(gf^{p^e-1})^{[p^e]}$, or equivalently, that $y^{p-1} \notin I_e(gf^{p^e-1})$. In view of Remark \ref{remarkI}, we look at terms in the product $gf^{p^e-1}$ that contribute something of the form $y^n$ to $I_e(gf^{p^e-1})$. A term does this if and only if the exponent for $x$ is strictly lower than $p^e$. In Equation \eqref{eq:gfpe-1} above, this happens for terms with index $i$ for which $i(p+1) + 1 \leq p^e - 1$, which is equivalent to
\[ i \leq \floor{\frac{p^e - 2}{p+1}} = \frac{p^e - p - 2}{p + 1}, \]
where we used again that $e$ is even. But for such indices $i$, the exponent for $y$ is given by
\[ (p^e - 1 - i)(p+1) \leq p^{e+1} + p^e - p - 1 - p^e + p + 2 = p^{e+1} + 1.\]
So the contribution of these terms to $I_e(gf^{p^e-1})$ is at least $y^p$. Thus the lowest exponent $n$ such that $y^n \in I_e(gf^{p^e-1})$ is $n = p$, and in particular $y^{p-1} \notin I_e(gf^{p^e-1})$. 
\end{proof}

\subsection{Some open questions}

\begin{quo}
The following questions are open, to the best of our knowledge.

\begin{enumerate}[(i)]

\item Does an algorithm exist which, on input polynomials $f$ and $g$, decides whether $\level (g,f)<\infty$? 

\item Under which conditions one can ensure that $\level (g,f)\leq\level (f)$?




\item In \cite[Proposition 6]{Fordham18}, it is shown that, if $R$
is an $F$--finite ring of characteristic $p\geq 3$, $f\in R$, and
$e$ is the largest $F$--jumping number of $f$ that lies inside
$(0,1)$, then $\level (f)=\lceil 1-\log_p (1-e)\rceil$. Is it possible
to obtain a similar result for $\level(g,f)$?

\end{enumerate}
\end{quo}

\bibliographystyle{alpha}
\bibliography{AFBoixReferences}

\end{document}